\documentclass[12pt,reqno]{amsart}
\usepackage{amsmath}
\usepackage{amssymb}

\newcommand{\C}{{\mathbb  C}}

\newcommand{\R}{{\mathbb  R}}

\newcommand{\N}{{\mathbb  N}}

\newcommand{\T}{{\mathbb  T}}

\newcommand{\F}{{\mathcal F}}

\renewcommand{\L}{{\mathcal L}}

\newcommand{\OO}{{\mathcal O}}

\newcommand{\boldcdot}{{\mathbf \cdot}}

\newcommand{\scalar}[2]{{\langle#1,#2\rangle}}
\newcommand{\PSH}{{\operatorname{{\mathcal {PSH}}}}}

\renewcommand{\Im}{{\operatorname{Im}}}
\renewcommand{\Re}{{\operatorname{Re}}}

\newcommand{\supp}{{\text{supp}\, }}
\newcommand{\chsupp}{{\text{ch\,supp}\, }}

\numberwithin{equation}{section}

\newtheorem{theorem+}           {Theorem}      [section]
\newtheorem{definition+}  [theorem+]  {Definition}
\newtheorem{lemma+}  [theorem+]  {Lemma}
\newtheorem{corollary+}  [theorem+]  {Corollary}
\newtheorem{proposition+}  [theorem+]  {Proposition}
\newtheorem{example+}  [theorem+]  {Example}

\newenvironment{theorem}{\begin{theorem+}\sl}{\end{theorem+}\rm}

\newenvironment{corollary}{\begin{corollary+}\sl}{\end{corollary+}\rm}
\newenvironment{proposition}{\begin{proposition+}\sl}{\end{proposition+}\rm}

\title[
Siciak's extremal functions and Helgason's support theorem]
{Siciak's homogeneous extremal functions,  
holomorphic extension  and  a generalization of 
Helgason's support theorem}
\author{J\"oran Bergh and Ragnar Sigurdsson}
\date{\today}
\dedicatory{{\small \it Dedicated to the memory of Professor J\'ozef Siciak}}

\begin{document}
\maketitle

\begin{abstract}
\noindent 
We prove that a function, which is defined on a union 
of lines $\C E$  through the origin in $\C^n$ with direction
vectors in $E\subset \C^n$
and is holomorphic 
of fixed finite order and finite type along each line, 
extends to an entire   holomorphic function 
on  $\C^n$ of the same order and finite type,
provided that $E$ has positive homogeneous capacity
in the sense of Siciak 
and all directional derivatives along the lines
satisfy a necessary compatibility condition at the origin.
We are able to estimate the indicator function of
the extension in terms of Siciak's weighted 
homogeneous extremal function, where the weight 
is a function of the type of the given function on 
each given line.
As an application we prove a generalization of 
Helgason's support theorem by showing how the support 
of a continuous function with rapid decrease at infinity
can be located from partial information on the support
of its Radon transform.
\end{abstract}

\section{Introduction}
\label{sec:1}

\noindent
This study grew out of the problem of locating the support of
a continuous function $u$ on $\R^n$ with rapid decrease at 
infinity   
from partial information of  the support of the Radon transform
$(\omega,p)\mapsto {\mathcal R}u(\omega,p)$, which is  
defined for $(\omega,p)\in {\mathbb S}^{n-1}\times \R$
as the integral of $u$  over  the hyperplane given
by the equation $\scalar x\omega=p$ 
with respect to the Lebesgue measure.  More precisely,
we assume that we have given a subset $E$ of ${\mathbb S}^{n-1}$ 
such that the convex hull of the support of the function 
$\R\ni p\mapsto {\mathcal R}u(\omega,p)$ is
contained in a closed bounded interval $[a_\omega,b_\omega]$ 
for every $\omega$ in  $E$ and from this information only 
we want to locate the support of $u$.  

Recall that Helgason's support theorem 
states that if $u\in C(\R^n)$ is rapidly decreasing, i.e., 
$|x|^k u(x)$ is bounded for every $k=1,2,3,\dots$, and  
there exists a compact convex subset $K$ of $\R^n$ 
with the property that 
${\mathcal R}u(\omega,p)=0$ for  every $(\omega,p)$ such that
the hyperplane defined by the equation $\scalar x\omega=p$ 
does not intersect $K$, then the support of $u$ is contained 
in $K$.  See Helgason \cite{Helgason:2011}, Th.~I.2.6 and Cor.~I.2.8.

Our result, Theorem \ref{th:6.1}, 
states  that if the set $E\subseteq {\mathbb S}^{n-1}$ 
is compact with positive homogeneous capacity in 
$\C^n$ in the sense of Siciak, 
e.g., if $E$ has non-empty interior in ${\mathbb S}^{n-1}$,
and 
$\sigma\colon E\to \R_+$ defined by
$\sigma(\omega)
=\max\{-a_\omega,b_\omega\}$ for $\omega\in E$, 
is bounded above, then
the support of $u$ is contained in the compact 
convex set
\begin{equation}
\label{eq:1.1}
\{x\in \R^n \,;\, \scalar x\omega\leq  \Psi_{E,\sigma}^*(\omega), 
\forall \omega\in {\mathbb S}^{n-1}\},
\end{equation}
where $\Psi_{E,\sigma}^*$ is the 
{\it upper semi-continuous regularization}  of  {\em Siciak's 
homo\-geneous extremal  function $\Psi_{E,\sigma}$ with weight}
$\sigma$, defined on $\C^n$ by
$$
\Psi_{E,\sigma}(\zeta)
=\sup\{ |p(\zeta)|^{1/k}\,;\, p\in {\mathcal P}^h(\C^n), \ 
k=\deg p\geq 1, \ |p|^{1/k}\leq \sigma \text{ on } E\}, 
$$
and  ${\mathcal P}^h(\C^n)$ 
is the set of homogeneous polynomials of $n$ complex variables. 
Furthermore, if $E$ is the closure of its relative interior in 
${\mathbb S}^{n-1}$ and the functions $\omega\mapsto a_\omega$ and
$\omega\mapsto b_\omega$ are lower and upper semicontinuous
functions on $E$, respectively, then the support of $u$ is contained
in
\begin{equation}
\label{eq:1.1a}
\{x\in \R^n \,;\, \scalar x\omega\leq  \Psi_{E,\sigma}^*(\omega), 
\forall \omega\in {\mathbb S}^{n-1}, 
 a_\omega\leq \scalar x\omega \leq b_\omega,
\forall \omega\in E\}.
\end{equation}

If we take the Fourier transform of the function 
${\mathcal R}u(\omega,\cdot)$, then we get the formula 
$\F_1\big({\mathcal R}u(\omega,\cdot)\big)(s)
=\F_nu(s\omega)=\widehat u(s\omega)$, 
where $\F_1$  and $\F_n$ are the Fourier transformations on 
$\R$ and  $\R^n$, respectively.  Since $u$ is rapidly decreasing,
we have $\widehat u\in C^\infty(\R^n)$ and  the chain rule gives
\begin{equation}
\label{eq:1.2}
  \dfrac 1{k!}\boldcdot 
\dfrac{d^k}{ds^k}
\F_1\big({\mathcal R}u(\omega,\cdot)\big)(s)\bigg|_{s=0}
=\sum_{|\alpha|=k}\dfrac{\partial^\alpha \widehat u(0)}{\alpha!} \omega^\alpha,
\qquad  \omega\in E,
\end{equation}
with multi-index notation:
$\alpha\in \N^n$, $\N=\{0,1,2,\dots\}$, 
$|\alpha|=\alpha_1+\cdots+\alpha_n$, 
$\alpha!=\alpha_1!\cdots\alpha_n!$, 
$s^\alpha=s_1^{\alpha_1}\cdots s_n^{\alpha_n}$,
$\partial^\alpha \widehat u(\xi)=
(\partial^{\alpha_1}_1\cdots 
\partial^{\alpha_n}_n)\widehat u(\xi)$,
and $\partial_j=\partial/\partial{\xi_j}$.

For every $\omega\in E$ the function
${\mathcal R}u(\omega,\cdot)$ has compact support, so the function
$\F_1\big({\mathcal R}u(\omega,\cdot)\big)$
extends to an entire function  of exponential type on $\C$.
The family 
$\big(\F_1\big({\mathcal R}u(\omega,\cdot)\big)\big)_{\omega\in E}$
defines a function $f\colon\C E\to \C$, by 
$f(z\omega)=\F_1\big({\mathcal R}u(\omega,\cdot)\big)(z)$ for 
$z\in \C$ and $\omega\in E$, and it satisfies 
\begin{equation}
\label{eq:1.3}
  |f(z\omega)|\leq \|{\mathcal R}u(\omega,\cdot)\|_{L^1(\R)}
e^{\sigma(\omega)|z|}, \qquad z\in \C.
\end{equation}

This leads us to a general problem of holomorphic extension.   
Assume that we have
given a subset $E$ of $\C^n$ and  
a function $f\colon \C E\to \C$, which is holomorphic
along each of the lines $\C\zeta$  and of 
type $\sigma(\zeta)$ with respect to the growth order $\varrho>0$ 
with uniform estimates, 
\begin{equation}
\label{eq:1.4}
  |f(z\zeta)|\leq Ce^{\sigma(\zeta)|z|^\varrho}, \qquad z\in \C, \ \zeta\in E.
\end{equation}
We would like to know under which conditions on $E$ it is possible to
extend $f$ to an entire function on $\C^n$ with similar growth
estimates.  
It is necessary to impose
some condition at the origin where all the different lines $\C\zeta$
intersect, for if $F$ is some holomorphic extension of $f$ to
a neighborhood of the origin, then $F$ links together the
power series of the functions $\C\ni z\mapsto f(z\zeta)$ at
the origin, because the chain rule implies 
\begin{equation*}
  \dfrac 1{k!}\boldcdot \dfrac{d^k}{dz^k}f(z\zeta)\bigg|_{z=0}
=\sum_{|\alpha|=k}\dfrac{\partial^\alpha F(0)}{\alpha!} \zeta^\alpha,
\qquad  \zeta\in E.
\end{equation*}
Our main result, Theorem \ref{th:4.1},  
states that if we assume that 
$E$ is compact with positive homogeneous capacity
in the sense of Siciak and $f$ satisfies
(\ref{eq:1.4}) together with 
a certain compatibility condition at the origin, then $f$ 
extends to an entire function on $\C^n$ with the growth
property that for every
$\varepsilon>0$ there exists a constant $C_\varepsilon>0$ 
such that 
\begin{equation}
\label{eq:1.5}
  |f(\zeta)|
\leq C_\varepsilon 
e^{\Psi_{E,\gamma}(\zeta)^\varrho+\varepsilon|\zeta|^\varrho},
\qquad \zeta\in \C^n,
\end{equation}
where $\gamma=\sigma^{1/\varrho}$.  As a consequence we
get the inequality $i_f\leq \Psi_{E,\gamma}^\varrho$, where
$i_f$ is the indicator function of $f$.

In the special case $f(z\omega)=\F_1\big({\mathcal
  R}u(\omega,\cdot)\big)(z)$
above we have $\varrho=1$. 
The indicator function of $\widehat u$ 
satisfies $i_{\widehat u}(\zeta)\leq i_{\widehat u}(i\Im \zeta)$
and the supporting function $H$  of $\chsupp u$ is given
on  $\R^n$ by $H(\xi)=i_{\widehat u}^*(i\xi)$.
The inequality  $H(\xi)\leq  \Psi_{E,\sigma}^*(\xi)$ enables
us to locate the support of $u$ in (\ref{eq:1.1}).
We have $i_{\widehat u}(-i\omega)\leq -a_\omega$ 
and $i_{\widehat u}(i\omega)\leq b_\omega$ for every
$\omega\in E$. In order to be able to conclude that 
$H(-\omega)\leq -a_\omega$ 
and $H(\omega)\leq b_\omega$
we need some regularity of $E$ 
and $\omega\mapsto (a_\omega,b_\omega)$.  Then
(\ref{eq:1.1a}) holds.

The plan of the paper is as follows.  In Section 2 we 
review a few facts on Siciak's extremal functions
and give some examples.   In Section 3 we review 
a few results on growth properties of entire functions
of one complex variable to be used later on.
In Section 4 we state and prove our main result
on the extension of a function on $\C E$ with growth
estimates to an entire function with similar growth
estimates on $\C^n$.  In Section 5 we review 
the variants of Paley-Wiener theorems which hold
for $L^2$ functions with compact support, distributions with compact
support, hyperfunctions with support in $\R^n$,
and analytic functionals on $\C^n$, and show
how the estimates proved in Section 4 can be used to 
locate supports and carriers.  In Section 6 we finally
prove the formulas (\ref{eq:1.1}) and (\ref{eq:1.1a}) 
for the location of the support of $u$ from partial 
information of the support
of $\R\ni p \mapsto {\mathcal R}u(\omega,p)$.

\noindent
{\bf Acknowledgement:}  A preliminary version
of this paper was presented  in the Siciak seminar at the Jagiellonian 
University in Cracow in 2012,  while the second author 
was a visiting researcher there.  
We are thankful to Professor Siciak for his interest in this work
and his advice, which led to important improvements of the paper.
We thank Miros{\l}aw Baran in Cracow,
Jan Boman in Stockholm, and Jan Wiegerinck in Amsterdam 
for helpful discussions.

\section{Siciak's extremal functions}
\label{sec:2}

\noindent
In this section we have collected a few results on 
extremal plurisubharmonic functions and capacities, which are related
to our results.
First a few words on the notation.  We use $\scalar\boldcdot\boldcdot$
both for the euclidean 
inner product on $\R^n$ and the natural bilinear form on $\C^n$,
$$
\scalar x\xi=\sum_{j=1}^nx_j\xi_j, \quad x,\xi \in \R^n, \quad
\scalar z\zeta=\sum_{j=1}^nz_j\zeta_j , \quad z,\zeta \in \C^n.
$$
Then the
hermitian form is $(z,\zeta)\mapsto \scalar z{\bar\zeta}$
and the euclidean norm is $\zeta\mapsto 
|\zeta|=\scalar \zeta{\bar\zeta}^{1/2}$.  

We let $\OO(X)$ denote space of all
holomorphic functions on an open subset $X$ of $\C^n$ and
$\PSH(X)$ the set of all plurisubharmonic functions on
$X$ which are not identically $-\infty$ in any connected component
of $X$. 

We let $\L=\L(\C^n)$ denote the set of all $u\in \PSH(\C^n)$
satisfying 
$$
u(\zeta)\leq \log^+|\zeta|+c_u, \qquad \zeta\in \C^n,
$$
for some constant depending on $u$, and we let $\L^h=\L^h(\C^n)$
denote the set of all $u$ in $\L$ which are {\it logarithmically
  homogeneous}, i.e., $u(t\zeta)=\log|t|+u(\zeta)$ for every
$t\in \C$ and $\zeta\in \C^n$.

We let  ${\mathcal P}(\C^n)$ denote the space of polynomials in 
$n$ complex variables  and ${\mathcal P}^h(\C^n)$ denotes the 
subset of homogeneous polynomials.  For every subset $E$ of
$\C^n$ and every function $\gamma\colon E\to \R_+$ we
define {\it Siciak's extremal function with weight $\gamma$}
on $\C^n$ by
$$
\Phi_{E,\gamma}(\zeta)=\sup\{ |p(\zeta)|^{1/k}\,;\, 
p\in {\mathcal P}(\C^n), \ 
k=\deg p\geq 1, \ |p|^{1/k}\leq \gamma \text{ on } E\},
$$
and  {\it Siciak's  homogeneous extremal function with weight
  $\gamma$} by
$$
\Psi_{E,\gamma}(\zeta)=
\sup\{ |p(\zeta)|^{1/k}\,;\, p\in {\mathcal P}^h(\C^n),
\  k=\deg p\geq 1, \ |p|^{1/k} \leq \gamma \text { on } E\}.
$$
In the special case $\gamma=1$ we denote these functions 
by $\Phi_E$ and $\Psi_E$ and call them 
{\it Siciak's   extremal function of the set $E$} and
{\it Siciak's   homogeneous extremal function of the set $E$},
respectively.  
We observe that the functions 
$\Psi_{E,\gamma}$ are {\it absolutely homogeneous of degree $1$},
i.e., $\Psi_{E,\gamma}(t\zeta)=|t|\Psi_{E,\gamma}(\zeta)$ for every
$t\in \C$ and $\zeta\in \C^n$.  We also observe that
if $\gamma$ is bounded above on $E$, $\gamma\leq \gamma_m$, for some
constant $\gamma_m$, then $\Phi_{E,\gamma}\leq \gamma_m \Phi_E$ and
$\Psi_{E,\gamma}\leq \gamma_m \Psi_E$.

Recall that a subset $E$ of $\C^n$ is said to be pluripolar
if every $a\in E$ has a connected neighborhood $U_a$ and 
$v\in \PSH(U_a)$  not identically $-\infty$ such that
$E\cap U_a\subseteq \{z\in U_a \,;\, v(z)=-\infty\}$.
If $E$ is not pluripolar we say that $E$ is {\it non-pluripolar}.
Josefson \cite{Jos:1978} proved that every pluripolar set $E$ is contained
in a set $\{\zeta\in \C^n\,;\, u(\zeta)=-\infty\}$ for some
$u\in \PSH(\C^n)$ and Siciak \cite{Sic:1982} proved that $u$ can
even be chosen in $\L$.

For every function $\varphi:E\to \overline{\R}=\R\cup\{\pm\infty\}$  
we define the {\it Siciak-Zakharyuta function with weight $\varphi$}
by
$$
V_{E,\varphi}(\zeta)=\sup \{u(\zeta)\,;\ u\in \L, u|E\leq \varphi\}, \qquad
\zeta\in \C^n,
$$ 
and the {\it homogeneous Siciak-Zakharyuta function with weight $\varphi$}
by
$$
V_{E,\varphi}^h(\zeta)=\sup \{u(\zeta)\,;\ u\in \L^h, u|E\leq \varphi\}, \qquad
\zeta\in \C^n.
$$ 
If $\varphi$ is the constant function $0$, then we 
we call these functions the {\it Siciak-Zakharyuta 
function of the set $E$ } and the 
{\it homogeneous Siciak-Zakharyuta function of the set $E$}
and denote them by $V_E$ and $V_E^h$, respectively.  In general,
if the function $\varphi$ is bounded above by the constant
$\varphi_m$, then we have $V_{E,\varphi}\leq \varphi_m+V_E$
and  $V_{E,\varphi}^h\leq \varphi_m+V_E^h$. 
Siciak \cite{Sic:1982} and Zakharyuta \cite{Zak:1976} proved
that  $V_E=\log \Phi_E$ and $V_E^h=\log \Psi_E$
for every compact subset of $\C^n$.

Let $\|\cdot\|$ be a complex norm on $\C^n$.  
For every  compact subset $E$ of $\C^n$ we define
the {\it capacity of $E$ in the sense of Siciak} by
$$
\varrho(E)= \exp\big(-\sup_{\|\zeta\|=1} V_E^*(\zeta)\big)
=\big(\sup_{\|\zeta\|=1} \Phi_E^*(\zeta)\big)^{-1}, 
$$
and the {\it homogeneous capacity of $E$  in the sense of Siciak} by
$$
\varrho^h(E)= \exp\big(-\sup_{\|\zeta\|=1} V_E^{h*}(\zeta)\big) 
=\big(\sup_{\|\zeta\|=1} \Psi_E^*(\zeta)\big)^{-1}.
$$
The compact subset $E$ of $\C^n$ 
is pluripolar if and only if $V_E^*\equiv +\infty$ if and only
if $\varrho(E)=0$.  (See Klimek \cite{Kli:1991}.)  Observe that
the definitions of the set functions $\varrho$ and $\varrho^h$ 
are depending on choice of the norm. 
The following result of Siciak \cite{Sic:1982}, Th.~1.10, gives
a geometric description of the homogeneous capacity:

\begin{theorem}
  \label{th:2.1} 
Let $\|\cdot\|$ be a complex norm on $\C^n$ with unit ball $B$ 
and $\varrho^h$ be the corresponding homogeneous capacity.  
Let $E$ be a  compact set in $\C^n$ and let 
$\widehat E=\{ z\in \C^n \,;\, \Psi_E^*(z)<1\}$ be the
{\it homogeneous hull of $E$}.  Then 
$$\varrho^h(E)=\sup\{r\in \R_+ \,;\, rB\subset  \widehat E\}.
$$ 
\end{theorem}

Since all norms on $\C^n$ are equivalent the property of
having  zero or strictly positive homogeneous  capacity 
is independent of the choice of norm.
We have ${\mathbb S}^{n-1} \subset
\{z\in \C^n \,;\, 
\log|z_1^2+\cdots+z_n^2-1|=-\infty\}$,
so $\varrho({\mathbb  S}^{n-1})=0$.
The following was proved by Korevaar \cite{Kor:1986}.

\begin{theorem}
  \label{th:2.2}
For every compact subset $E$ of\,    
$\, {\mathbb S}^{n-1}$ with non-empty relative interior $\varrho^h(E)>0$. 
\end{theorem}

Let $E\subset \C^n$ and 
let  $\T$ denote the unit circle in $\C$.
We define then {\it circular hull} of $E$ by  
$E_c=\T E=\{tz\,;\, t\in \T, z\in E\}$, 
and we say that $E$ is {\it circular} if $E_c=E$.
Since  $\sup_{E_c}|p|=\sup_E|p|$ for every homogeneous
polynomial $p$, it is clear that $\Psi_{E_c}=\Psi_E$.
Furthermore, if $E_c$ is non-pluripolar
then
$\Psi_E^*\in \PSH(\C^n)$ 
and $\Psi_E^*$ is absolutely homogeneous of degree $1$,  
i.e.,
$$
\Psi_E^*(z\zeta)=|z| \Psi_E^*(\zeta), \qquad \zeta\in \C^n, \ z\in \C.
$$
The following is a result of Siciak \cite{Sic:1981}

\begin{theorem} \label{th:2.3}
If $E$ is circular compact set, then its  polynomial hull is
$$
\widehat E=\{\zeta\in \C^n\,;\, \Psi_E(\zeta)\leq 1 \}.
$$
\end{theorem}

There are very few explicit formulas for $\Psi_E$. 
The most important is:

\begin{proposition} \label{prop:2.4}
If $E$ is the unit ball for a complex norm $\|\boldcdot\|$ on
$\C^n$, then
$$
\Psi_E(\zeta)=\|\zeta\|, \qquad \zeta\in \C^n.
$$  
\end{proposition}

If  $\|\boldcdot\|$ is a norm on $\R^n$,   then the largest complex
norm on $\C^n$ which extends $\|\boldcdot\|$ is the {\it cross norm}
$\|\boldcdot\|_c$.  It is given by the formula
$$
\|\zeta\|_c=\inf\{\sum_{j=1}^N|\alpha_j|\|\omega_j\|\,;\, 
\zeta=\sum_{j=1}^N \alpha_j\omega_j, \alpha_j\in \C, \omega_j\in \R^n\}.
$$
For a proof of the following result see Siciak \cite{Sic:1961}
and Dru{\.z}kowski \cite{Dru:1974}.
 
\begin{proposition}
  \label{prop:2.5}
If $E\subset \R^n$ is the unit sphere  in the norm $\|\boldcdot\|$, then
$$
\|\zeta\|_c \leq \Psi_E(\zeta), \quad \zeta\in \C^n \text{ and } \quad
\|\zeta\|_c =\Psi_E(\zeta), \quad \zeta\in \C\R^n
$$
\end{proposition}

If the norm $\|\cdot\|$ in $\R^n$ is given by an inner product,
then we have equality:

\begin{theorem} \label{th:2.6} If $x,\xi\mapsto \scalar x\xi$ 
is an inner product on $\R^n$, $\xi\mapsto |\xi|=\scalar \xi\xi^{\frac 12}$
is the corresponding norm and $E=\{\xi\in \R^n\,;\, |\xi|=1\}$ 
is the unit sphere, then 
$\Psi_E=|\cdot|_c$.  
 If $\zeta=\xi+i\eta=e^{i\theta}(a+ib)$, 
$\theta\in \R$,  $\xi,\eta,a,b\in \R^n$, 
$\scalar ab=0$, and $|b|\leq |a|$, then 
\begin{equation*}
  |a|=\tfrac 1{\sqrt 2}\big(|\zeta|^2+|\scalar \zeta\zeta|\big)^{\frac 12}
\quad \text{ and } \quad
  |b|=\tfrac 1{\sqrt 2}\big(|\zeta|^2-|\scalar \zeta\zeta|\big)^{\frac 12},
\end{equation*}
and 
\begin{align*}
|\zeta|_c&=|a|+|b|
= \big(|\zeta|^2-d(\zeta,\C \R^n)^2\big)^{\frac 12}+d(\zeta,\C \R^n)\\
&= \big(|\zeta|^2 +\big(
|\zeta|^4-|\scalar \zeta\zeta|^2\big)^{\frac 12}\big)^{\frac 12}
\\
&=\big(|\zeta|^2 +2\big(|\xi|^2|\eta|^2-\scalar \xi\eta^2\big)^{\frac
  12}\big)^{\frac 12}.
\end{align*}
As a consequence we have $|\zeta|\leq |\zeta|_c\leq \sqrt 2|\zeta|$
for every $\zeta\in \C^n$, we have $|\zeta|=|\zeta|_c$ if and only if
$\zeta\in \C\R^n$, and we have $|\zeta|_c=\sqrt 2 |\zeta|$ if and
only if $\scalar \zeta\zeta=\zeta_1^2+\cdots+\zeta_n^2=0$.
\end{theorem}

See Dru{\.z}kowski \cite{Dru:1974} and Sigurdsson and Sn{\ae}bjarnarson
\cite{SigSna:2018} for a proof.  
As a consequence of Theorems \ref{th:2.1} and \ref{th:2.6} 
we have:

\begin{corollary}
  \label{cor:2.7}  If $|\cdot|$ is a norm with respect to an inner
  product on $\R^n$, then the corresponding homogeneous capacity
of the unit sphere is $1/\sqrt 2$.
\end{corollary}

For the case when $E$ is the unit ball with respect to a 
norm in $\R^2$ we have a very interesting result of 
Baran \cite{Bar:1999}:

\begin{proposition}
Let $\|\boldcdot\|$ be a norm on $\R^2$ with unit ball $E$, set 
$u(\xi)=\log\|(1,\xi)\|$ for $\xi$, and define
$$
{\mathcal P}u(z)=\dfrac{|\Im z|}\pi \int_\R
\dfrac{u(\xi)\, d\xi}{|z-\xi|^2}, \qquad z\in \C\setminus \R
$$
Then 
$$
\Psi_E(z_1,z_2)=
  |z_1|\exp\big({\mathcal P}u(z_2/z_1)\big), \qquad z\in \C^2, \ z_1\neq 0,
$$
and it extends uniquely to a continuous function on $\C^2$.
\end{proposition}

As we have already mentioned there are only a few explicit examples of
homogeneous extremal functions.  
{\it Disc envelope formulas} are an alternative way of expressing
extremal functions, see
L\'arusson and Sigurdsson \cite{LarSig:2005,LarSig:2007,LarSig:2009},
Magn\'usson and Sigurdsson \cite{MagSig:2007}, and
Drinovec Drnov{\v s}ek and Sigurdsson \cite{DriSig:2016}.

We have only mentioned results on extremal functions that
are needed for proving our results.  For the general theory
see the works of 
Siciak \cite{Sic:1961,Sic:1962,Sic:1974,Sic:1981,Sic:1982},
Baran \cite{Bar:1998,Bar:1999,Bar:2009}, and
Klimek \cite{Kli:1991}.
Some interesting applications are given in
Korevaar \cite{Kor:1986,Kor:1986a}.
For a related capacity on projective spaces see
Alexander \cite{Ale:1981,Ale:1982}.

\section{Growth properties entire functions}
\label{sec:3}

\noindent
We say that a function  $f\in \OO(\C^n)$ is of {\it finite order}
if there exist positive constants $C$, $\sigma$ and $\varrho$ such 
that 
\begin{equation}
  \label{eq:3.1}
 |f(\zeta)|\leq Ce^{\sigma |\zeta|^\varrho}, \qquad \zeta\in \C^n,  
\end{equation}
and we define the {\it order} $\varrho_f$ of $f$ as the infimum over 
all $\varrho$ for which such an estimate exists.  If  $f$ is not of
finite order, then we say that $f$ is of {\it infinite order} and
define $\varrho_f=+\infty$. If we set
$M_f(r)=\sup_{|\zeta|\leq r} |f(\zeta)|$, 
then
$$
\varrho_f=\varlimsup_{r\to +\infty}\dfrac{\log\log M_f(r)}{\log r}.
$$
If $f$ is of finite order, then we say that $f$ is of {\it finite
type with respect to the growth order $\varrho$} 
if there exist positive constants $C$ and $\sigma$ 
such that (\ref{eq:3.1}) holds.
Then we define the type {\it type} $\sigma_f$ of
$f$ ({\it with respect to the order $\varrho$}) as the infimum over
all $\sigma$ for which (\ref{eq:3.1}) holds for some $C$.
We have
$$
\sigma_f=\varlimsup_{r\to +\infty}\dfrac{\log M_f(r)}{r^{\varrho}}.
$$
The function $f$ is said to be of {\it exponential type} if it is of
finite type with respect to the growth order $1$. 

\smallskip
For every $f\in\OO(\C^n)$ of finite type with respect to the 
order $\varrho$ we define the {\it indicator function}
$i_f$  by 
$$
i_f(\zeta)=\varlimsup_{t\to \infty}
\dfrac 1{t^{\varrho}}\log|f(t\zeta)|, 
$$
then its {\it upper semi-continuous regularization},
$$
i_f^*(\zeta)=\varlimsup_{\vartheta\to \zeta}
i_f(\vartheta),  \qquad \zeta \in \C^n,
$$
is plurisubharmonic and we have
$$
\sigma_f=\sup_{|\zeta|=1} i_f^*(\zeta).
$$

Let $\varphi\in \OO(\C)$ be given by 
$\varphi(z)=\sum_{k=0}^\infty c_kz^k$ and assume that
$|\varphi(z)|\leq Ce^{\sigma |z|^\varrho}$ for
all $z\in \C$, where $C$, $\varrho$, and $\sigma$ 
are positive constants.
Then Cauchy's inequalities give that for every $r>0$ we have
$$|c_k|=\dfrac{|\varphi^{(k)}(0)|}{k!}\leq C\dfrac{e^{\sigma r^\varrho}}{r^k},
\qquad k=0,1,2,\dots.
$$ 
The minimal value of the right hand side is taken for 
$r=(k/\sigma\varrho)^{1/\varrho}$, so we conclude that
\begin{equation}
  \label{eq:3.2}
|c_k|
\leq C\bigg(\dfrac{e\sigma\varrho}k\bigg)^{k/\varrho},
\qquad k\geq 0,
\end{equation}
(with the abuse of notation $(e\sigma\varrho/0)^{0/\varrho}=1$), 
and 
\begin{equation}
  \label{eq:3.3}
|\varphi(z)|\leq
C\sum_{k=0}^\infty\bigg(\dfrac{e\sigma\varrho}k\bigg)^{k/\varrho}
|z|^k, \qquad z\in \C.  
\end{equation}
By Levin \cite{Lev:1980}, Ch.~1., it is possible to express the 
order and type of any $\varphi\in \OO(\C)$  
in terms of the coefficients $c_k$ of its power series at
$0$. Its order $\varrho_\varphi\in [0,+\infty]$ is given 
by the formula
$$
\varrho_\varphi=\varlimsup_{k\to \infty}
\dfrac{k\log k}{-\log|c_k|}.
$$
If $\varrho_\varphi\in ]0,+\infty[$, then the type
$\sigma_\varphi\in [0,+\infty]$ with respect to 
$\varrho_\varphi$ is given by the equation
$$
\big(e \sigma_\varphi \varrho_\varphi\big)^{\tfrac 1\varrho_\varphi}
=\varlimsup_{k\to \infty} k^{\tfrac 1\varrho_\varphi}\root k \of {|c_k|}.
$$
These formulas tell us that for any given positive numbers
$\varrho$ and $\sigma$ the entire function 
\begin{equation}
  \label{eq:3.4}
z\mapsto \sum_{k=0}^\infty\bigg(\dfrac{e\sigma\varrho}k\bigg)^{k/\varrho}
z^k, \qquad z\in \C,
  \end{equation}
is of order $\varrho$ and type $\sigma$ with respect to $\varrho$.

\section{Extensions of holomorphic functions}
\label{sec:4}

\noindent
Our main result is 

\begin{theorem}
  \label{th:4.1}
Let $E$ be a compact subset of $\C^n$ and assume that
$E$ has positive homogeneous capacity in the sense of 
Siciak.
Let  $f\colon \C E \to \C$ be a function,
such that for every $\zeta\in E$ the function
$\C\ni z\mapsto f(z\zeta)$ is holomorphic and satisfies
\begin{equation}
  \label{eq:4.1}
  |f(z\zeta)|\leq Ce^{\sigma(\zeta)|z|^\varrho}, \qquad z\in \C, \ \zeta\in E, 
\end{equation}
with positive constants  $C$  and $\varrho$
and a function  $\sigma\colon E\to \R_+=\{ x\geq 0\}$, which is bounded above.
Assume that  for every $k=0,1,2,\dots$ there exists a 
$k$-homogeneous complex polynomial $P_k$ on $\C^n$ such that
\begin{equation}\label{eq:4.2}
  \dfrac 1{k!}\boldcdot \dfrac{d^k}{dz^k}f(z\zeta)\bigg|_{z=0}
=P_k(\zeta),
\qquad  k=0,1,2,\dots, \quad \zeta\in E.
\end{equation}
Then the series $\sum_{k=0}^\infty P_k$ converges locally uniformly
in $\C^n$ and gives a unique holomorphic
extension of $f$ to $\C^n$ by 
$f(\zeta)=\sum_{k=0}^\infty P_k(\zeta)$ for $\zeta\in \C^n$.
We have 
\begin{equation}
  \label{eq:4.3}
|f(\zeta)|\leq 
C\sum_{k=0}^\infty
\bigg(\dfrac{e\varrho}k\bigg)^{k/\varrho}
\Psi_{E,\gamma}(\zeta)^k, \qquad 
\zeta\in \C^n,
  \end{equation}
where $\Psi_{E,\gamma}$ is Siciak's weighted homogeneous
extremal  function with weight $\gamma=\sigma^{1/\varrho}$.
This  implies that for every $\varepsilon>0$ there exists
$C_\varepsilon>0$ such that
\begin{equation} 
  \label{eq:4.4}
|f(\zeta)|
\leq C_\varepsilon 
e^{\Psi_{E,\gamma}(\zeta)^\varrho+\varepsilon|\zeta|^\varrho},
\qquad \zeta\in \C^n.  
\end{equation}
Hence $f$ is of order $\leq \varrho$ and of type
$\leq \alpha^\varrho$ with respect to $\varrho$, where 
$\alpha=\sup_{|\zeta|=1}\Psi_{E,\gamma}^*(\zeta)$,
and $i_f\leq \Psi_{E,\gamma}^\varrho$.
\end{theorem}

\begin{proof} There is no condition ensuring convergence 
of the series  $\sum_{k=0}^\infty P_k$, so we show
that it is locally uniformly convergent.
For every $\zeta\in E$ we define $\varphi_\zeta\in
\OO(\C)$  by $\varphi_\zeta(z)=f(z\zeta)$ for $z\in \C$.  
By  (\ref{eq:4.1}) $\varphi_\zeta$ is of type $\leq \sigma(\zeta)$
with respect to the growth order $\varrho$, so 
(\ref{eq:3.2}) gives 
$$
|P_k(\zeta)|= 
\dfrac{|\varphi_\zeta^{(k)}(0)|}{k!} \leq 
C\bigg(\dfrac{e\sigma(\zeta)\varrho}k\bigg)^{k/\varrho}
=C\bigg(\dfrac{e\varrho}k\bigg)^{k/\varrho} \gamma(\zeta)^k, 
\qquad \zeta\in E.
$$
and by the definition of $\Psi_{E,\gamma}$ we have
\begin{equation}
  \label{eq:4.5}
{|P_k(\zeta)|}
\leq  C\bigg(\dfrac{e\varrho}k\bigg)^{k/\varrho} \Psi_{E,\gamma}(\zeta)^k,
\qquad \zeta\in \C^n.
\end{equation}
Since $\sigma$ is bounded above,
$\gamma \leq \gamma_m$  for some constant $\gamma_m$,
and we have $\Psi_{E,\gamma}\leq \gamma_m \Psi_E$.
Since $E$ has positive homogeneous capacity,
it now follows from Siciak's theorem that
both $\Psi_E^*$ and $\Psi_{E,\gamma}^*$ are 
plurisubharmonic and therefore
locally bounded above in $\C^n$.
Hence we conclude that $\sum_{k=0}^\infty P_k$ 
converges locally uniformly in 
$\C^n$ and from (\ref{eq:4.5}) we see that
 the limit defines a function 
$F\in \OO(\C^n)$ satisfying (\ref{eq:4.3}).
The function $F$ is an extension of $f$, for if $\zeta\in E$, then 
for every $z\in \C$
$$
F(z\zeta)=\sum_{k=0}^\infty P_k(z\zeta) 
=\sum_{k=0}^\infty P_k(\zeta)z^k
=\sum_{k=0}^\infty \dfrac{\varphi_\zeta^{(k)}(0)}{k!}z^k
=\varphi_\zeta(z)=f(z\zeta).
$$
Since the function (\ref{eq:3.4}) is 
of order $\varrho$ and of type $\sigma$ with respect to 
$\varrho$,  we even have (\ref{eq:4.4})
and  the inequality  $i_f\leq \Psi_{E,\gamma}^\varrho$
follows from homogeneity.
\end{proof}

It is interesting to state a special case of Theorem \ref{th:4.1} 
for functions of exponential type and
combine it with Theorem \ref{th:2.6} and Corollary \ref{cor:2.7}:

\begin{corollary}
  \label{cor:4.2} 
If the assumptions of Theorem \ref{th:4.1} are satisfied 
with $\varrho=1$ and $E\subset {\mathbb S}^{n-1}$, then
the extension $f$ satisfies
\begin{equation} 
  \label{eq:4.6}
|f(\zeta)|
\leq C_\varepsilon 
e^{\sigma_m|\zeta|_c+\varepsilon|\zeta|}
\leq C_\varepsilon 
e^{\sqrt 2\sigma_m|\zeta|+\varepsilon|\zeta|},
\qquad \zeta\in \C^n,  
\end{equation}
for every $\varepsilon>0$, 
where $\sigma_m=\sup_E\sigma$.  As a consequence 
we note that if the restriction of $f$ to the union
of lines $\C E\subseteq \C\R^n$ has type $\leq \sigma_m$, then
the extension is of type $\leq \sqrt 2\, \sigma_m$.
\end{corollary}

Extremal plurisubharmonic functions have been applied for 
extension of holomorphic functions in the same spirit
as Theorem \ref{th:4.1}.  For example the following 
result of Siciak \cite{Sic:1982}, Th.~13.4:

\begin{theorem}
  \label{th:4.3}
Let $E$ be a subset of $\C^n$ with $\varrho^h(E)>0$ 
and $P_k$ be a  homogeneous
polynomial on $\C^n$ of degree $k$ for $k=0,1,2,\dots$.
Assume that the series $\sum_{k=0}^\infty P_k$ converges
at every point of $E$ outside a subset of homogeneous
capacity zero. Then the series converges locally uniformly
in $\{z\in \C^n\,;\, \Psi_E^*(z)<1\}$.
\end{theorem}

For similar results see Forelli \cite{For:1977}
and Wiegerinck and Korevaar  \cite{WieKor:1985}.
Estimates of the growth of a plurisubharmonic function 
in $\C^n$ in terms of the  growth along a complex cone 
$\C E$ were proved in Sibony and Wong \cite{SibWon:1980}
with optimized constants in Siciak \cite{Sic:1982}, Cor.~11.2.
See also Korevaar \cite{Kor:1986,Kor:1986a}.

\section{Fourier-Laplace transforms and Paley-Wiener theorems}
\label{sec:5}

\noindent
The main motivation for studying entire functions 
is the fact that Fourier-Laplace transforms of analytic functionals,
hyperfunctions with compact support,  distributions with compact
support, and $L^2$-functions with compact support,
are entire functions of exponential type.  
For each of these 
classes of functionals and functions 
there is a variant of the Paley-Wiener theorem which describes
how  estimates of Fourier-Laplace transforms are used to 
locate hulls of carriers or supports.   
This is based on 
the fact that every convex function $H$ on $\R^n$, which is positively
homogeneous of degree~$1$,  is the supporting function  
$$
H(\xi)=\sup_{x\in K}\scalar x\xi, \qquad \xi\in \R^n,
$$ 
of a unique compact convex set $K$ which is related to $H$ by 
$$
K=\{x\in \R^n\,;\, \scalar x\xi\leq H(\xi), \forall \xi\in \R^n\}.
$$  
(For a proof see \cite{Hormander:LPDO}, Th.~4.3.2.) 
We define the Fourier transform of $u\in L^1(\R^n)$ 
by
$$
\widehat u(\xi)=\int_{\R^n} e^{-i\scalar x\xi}u(x)\, dx, \qquad \xi\in
\R^n.
$$
If $u$ has a compact support, i.e., $u$ vanishes almost everywhere
outside a compact set,  then its Fourier transform extends to an 
entire function on $\C^n$, which is called the {\it Fourier-Laplace
transform} of $u$ and is given by the formula
$$
\widehat u(\zeta)=\int_{\R^n} e^{-i\scalar x\zeta}u(x)\, dx, \qquad \zeta\in
\C^n,
$$
and if $K=\chsupp u$ with supporting function $H_K$, we have 
the estimate
$$
|\widehat u(\zeta)|\leq \int_K e^{\scalar x{\Im \zeta}} |u(x)|\, dx
\leq  \|u\|_{L^1(\R^n)}e^{H_K(\Im \zeta)},
$$
If $u\in L^1\cap L^2(\R^n)$, not necessarily with compact support,
then $\widehat u\in L^2(\R^n)$ and the Plancherel formula
$$
\|\widehat u\|_{L^2(\R^n)}=(2\pi)^{\frac n2}\|u\|_{L^2(\R^n)}
$$
implies that $u\mapsto \widehat u/(2\pi)^{\frac n2}$ extends 
from $L^1\cap L^2(\R^n)$ to an isometry on $L^2(\R^n)$.
If $u\in L^2(\R^n)$ has compact support and
$K=\chsupp u$ has supporting function $H_K$ 
then we have the growth estimate
$$
|\widehat u(\zeta)|\leq \int_K e^{\scalar x{\Im \zeta}} |u(x)|\, dx
\leq \lambda(K)^{\frac 12} \|u\|_{L^2(\R^n)}e^{H_K(\Im \zeta)},
$$
where $\lambda$ is the Lebesgue measure on $\R^n$.
The following is the original  Paley-Wiener theorem on $\R^n$:

\begin{theorem} \label{th:5.1}
The function $f\in \OO(\C^n)$ is the Fourier-Laplace
transform of an $L^2$-function with support contained in
the compact convex
subset $K$ of $\R^n$ if and only if 
$f|\R^n\in L^2(\R^n)$ and there
exists $C>0$ such that
$$
|f(\zeta)|\leq Ce^{H_K(\Im \zeta)},
\qquad \zeta\in \C^n.
$$
\end{theorem}

If $u\in {\mathcal E}'(\R^n)$ is a distribution with compact support, then
its Fourier transform $\widehat u$ is in $C^\infty(\R^n)$ and
it extends to an entire holomorphic function $\widehat u$ which
is given by 
$\widehat u(\zeta)= u(e^{-i\scalar\boldcdot\zeta})$ for $\zeta\in \C^n$,
i.e., we get $\widehat u(\zeta)$ by letting 
$u$ act on the $C^\infty$ function $x\mapsto
e^{-i\scalar x\zeta}$.  The following is called the
Paley-Wiener-Schwartz-theorem.
(For a proof see \cite{Hormander:LPDO}, Th.~7.3.1.)

\begin{theorem} \label{th:5.2}
The function $f\in \OO(\C^n)$ is the Fourier-Laplace
transform of a distribution with support contained in
the compact convex
subset $K$ of $\R^n$ if and only if there
exist $C>0$ and $N\geq 0$  such that
$$
|f(\zeta)|\leq C(1+|\zeta|)^N e^{H_K(\Im \zeta)},
\qquad \zeta\in \C^n.
$$
\end{theorem}

Recall that an {\it analytic functional} is an element
$u$ in $\OO'(\C^ n)$, i.e., a continuous linear functional acting on the
space of entire functions.
By continuity there exists a compact subset $M$ of $\C^n$ 
and a positive constant $C$ such that
\begin{equation}
  \label{eq:5.1}
|u(\varphi)|\leq C\sup_{\zeta\in M} |\varphi(\zeta)|, 
\qquad \varphi\in \OO(\C^n).  
\end{equation}
Observe that the Hahn-Banach theorem implies that 
$u$ extends to a linear  functional on $C(M)$ satisfying
(\ref{eq:5.1}) for $\varphi\in C(M)$. By the Riesz representation
theorem there exists a complex measure $\mu$ on $M$ such that
$$
u(\varphi)=\int_M \varphi \, d\mu, \qquad \varphi\in C(M).
$$
We say that $u$ is {\it carried by} the compact subset 
$K$ of $\C^n$ if for every neighbourhood $U$ of $K$ in $\C^n$
there exists a constant $C_U$ such that
$$
|u(\varphi)|\leq C_U \sup_{\zeta\in U} |\varphi(\zeta)|, 
\qquad \varphi\in \OO(\C^n).
$$
In this case we also say that $K$ is a {\it carrier} for $u$.
The Fourier-Laplace transform of $u\in \OO'(\C^n)$ is defined by
$$
\widehat u(\zeta)=u(e^{-i\scalar{\boldcdot}\zeta}), \qquad \zeta\in \C^n.
$$
We can differentiate with respect to $\zeta_j$ and $\bar \zeta_j$ 
under the $u$-sign and conclude that $\widehat u\in \OO(\C^n)$.
Furthermore, if $u$ is carried by the compact set $K$ then we take
$\varepsilon>0$ and 
$K_\varepsilon=K+B(0,\varepsilon)$ as $U$ in the definition of a
carrier and conclude that
$$
|\widehat u(\zeta)|\leq C_\varepsilon
e^{H_K(-i\zeta)+\varepsilon|\zeta|}, \qquad \zeta\in \C^n,
$$
where $H_K$ is the supporting function of the set $K$,
now defined as
$$
H_K(\zeta)=\sup_{z\in K}\Re \scalar z\zeta, \qquad \zeta\in \C^n.
$$
Observe that here  we use the real bilinear form 
$$(z,\zeta)\mapsto \Re\scalar z\zeta=\scalar x\xi-\scalar y\eta, \quad
z=x+iy, \zeta=\xi+i\eta\in \C^n
$$
instead of the euclidean inner product 
$$
(z,\zeta)\mapsto \Re\scalar {\bar z}\zeta=\scalar x\xi+\scalar y\eta, \quad
z=x+iy, \zeta=\xi+i\eta\in \C^n,
$$
which we use to identify $\C^n$ with the real euclidean 
space $\R^{2n}$. This means that if $H:\C^n\to \R$ is convex and positively 
homogeneous of degree 1, 
$L=\{z\in \C^n\,;\, \Re \scalar {\bar z}\zeta\leq H(\zeta)\}$
and $K=\{\bar z\,;\, z\in L\}$, then
$$
H_K(\zeta)=\sup_{z\in K}\Re\scalar z\zeta =H(\zeta), \qquad \zeta\in \C^n.
$$
The following variant of the Paley-Wiener theorem is usually called
the P\'olya-Ehrenpreis-Martineau theorem.  
(For a proof see \cite{Hormander:SCV}, Th.~4.5.3.)

\begin{theorem} \label{th:5.3}
The function $f\in \OO(\C^n)$ is the Fourier-Laplace
transform of an analytic functional carried by the compact convex
subset $K$ of $\C^n$ if and only if for every $\varepsilon>0$ there
exists a constant $C_\varepsilon$ such that
$$
|f(\zeta)|\leq C_\varepsilon e^{H_K(-i\zeta)+\varepsilon|\zeta|},
\qquad \zeta\in \C^n.
$$
\end{theorem}

For every subset $A$  of $\C^n$ we let $\OO'(A)$ denote the set of all
analytic functionals carried by a compact subset of $A$.  
Even if $u$ is carried by two compact subsets $K_1$ and $K_2$,
it does not mean that $u$ is carried by
$K_1\cap K_2$, i.e., in general analytic functionals do not 
have a unique minimal carrier.  
If, on the other hand, 
$u\in \OO'(\R^n)$ then by a theorem of Martineau,  $u$ has a minimal carrier in
$\R^n$ which is called the support of $u$ and is denoted by
$\supp u$. (For a proof see \cite{Hormander:LPDO}, Th.~9.1.6.)
The space $\OO'(\R^n)$ is identified with the space of hyperfunctions
with compact support.
The following theorem is a variant 
of the Paley-Wiener theorem for hyperfunctions. It
is sometimes  called the
Paley-Wiener-Martineau theorem. 
(For a proof see \cite{Hormander:LPDO}, Th.~15.1.5.)

\begin{theorem}
  \label{th:5.4}
The function $f\in \OO(\C^n)$ is the Fourier-Laplace
transform of a hyperfunction with support contained in
the compact convex
subset $K$ of $\R^n$ if and only if for every $\varepsilon>0$ there
exists $C_\varepsilon>0$ such that
$$
|f(\zeta)|\leq C_\varepsilon e^{H_K(\Im \zeta)+\varepsilon|\zeta|},
\qquad \zeta\in \C^n.
$$
\end{theorem}

If $u$ is a hyperfunction, distribution or an $L^2$ function 
with compact support and $i^*_{\widehat u}$ is the regularized 
indicator function of its Fourier-Laplace transform 
then the function 
$\R^n\ni \xi \mapsto i^*_{\widehat u}(i\xi)$,
is the supporting function $H_K$ of $\chsupp u$. 
(For a proof see \cite{Sig:1986}, Th.~2.1.1.)  
This implies that 
if we have a growth estimate of the form
$$|\widehat u(\zeta)|\leq C_\varepsilon
e^{\Psi(\zeta)+\varepsilon|\zeta|}, \qquad   \zeta\in \C^n,
$$
where $\Psi$ is a function on $\C^n$, which is absolutely 
homo\-geneous of degree $1$, i.e., $\Psi(t\zeta)=|t|\Psi(\zeta)$
for every $\zeta\in \C^n$ and $t\in \C$, 
 and $C_\varepsilon>0$ 
is a constant for every
$\varepsilon>0$, then
$$
\chsupp u 
\subseteq \{x\in \R^n\,;\, \scalar x\xi\leq \Psi^*(\xi), 
\forall \xi\in \R^n\}.
$$
Now we have reviewed all the Paley-Wiener theorems which
are relevant to our application of Theorem~\ref{th:4.1}:

\begin{theorem}\label{th:5.5}
Let $E$ be a compact subset of $\C^n$ 
with positive homo\-geneous capacity
and 
$f\colon \C E\to \C$ be a function satisfying the 
conditions in Theorem~\ref{th:4.1} with
$\varrho=1$.

(i) Then $f$ extends to an entire function on $\C^n$ 
which is the Fourier-Laplace transform of an analytic functional 
with a carrier contained in the ball with center at the origin
and radius $\alpha=\sup_{|\zeta|=1} \Psi_{E,\sigma}^*(\zeta)$.

(ii) If, in addition, there exists a constant $A>0$ such
that  $i_f^*(\zeta)\leq A|\Im \zeta|$ for every $\zeta\in \C^n$, 
then $f$ is the Fourier-Laplace transform of
a hyperfunction with support contained in the
compact convex set
$$K_{E,\sigma}=\{x\in \R^n \,;\, \scalar x\xi\leq  \Psi_{E,\sigma}^*(\xi), 
\forall \xi\in \R^n\}.$$

(iii) If we have the estimate $|f(\xi)|\leq C(1+|\xi|)^N$ for all
$\xi\in \R^n$, for some constants $C>0$ and $N\geq 0$,
then $f$ is the Fourier-Laplace transform 
of a distribution with compact support in $K_{E,\sigma}$.

(iv) Finally, if we can conclude that
$f$ is in $L^2(\R^n)$, then $f$ is the Fourier-Laplace transform 
of a function  in  $L^2(\R^n)$ with support in $K_{E,\sigma}$.
\end{theorem}

\begin{proof}  (i) By Theorem \ref{th:4.1} $f$ has an extension to an
entire function satisfying (\ref{eq:4.4}).  
The supporting function of the closed ball with center at the origin
and radius $\alpha$ is $\zeta\mapsto \alpha|\zeta|$, so (i) follows
from Theorem 5.3.  

(ii) By \cite{HorSig:1998}, Th.~3.9, it follows that
for every $\varepsilon>0$ there exists $C_\varepsilon>0$ such
that $f(\zeta)|\leq C_\varepsilon e^{A|\Im\zeta|+\varepsilon|\zeta|}$
for $\zeta\in \C^n$, and by Theorem \ref{th:5.4} the function $f$ is the
Fourier-Laplace transform of a hyperfunction $u$ with 
support contained in the ball $\{x\in \R^n\,;\, |x|\leq A\}$.
Furthermore, the supporting function of $K=\chsupp u$ 
is given by $H_K(\xi)=i_f^*(i\xi)$ for $\xi\in \R^n$.
Since $i_f^*\leq \Psi_{E,\sigma}^*$, we have $H_K(\xi)\leq
\Psi_{E,\sigma}^*(\xi)$ for every 
$\xi\in \R^n$ and it follows that $K\subseteq K_{E,\sigma}$.

(iii) Recall \cite{HorSig:1998}, Lemma 2.1, which is an
application of the Phragm\'en-Lindel\"of principle:  
Let $v$ be a subharmonic function in the upper half plane
$\C_+=\{z\in \C \,;\, \Im z>0\}$ such that for some real constants
$c$ and $a$ we have 
$v(z)\leq c+ a|z|$ for every $z\in \C_+$ and
$\varlimsup_{\C_+\ni z\to x}v(z)\leq 0$ 
for every $x\in \R$.  Then $v(z)\leq a\, \Im z$  for every 
$z\in \C_+$.

We are now going to apply this lemma to the function 
$$
v(z)=\log|f(\xi+z\eta)|-\log C -\tfrac 12 N\log 2
-N\log|1+|\xi|-iz|\eta||, \qquad z\in \C_+.
$$
It is subharmonic for $\log|f|$ is plurisubharmonic 
on $\C^n$ and the logarithmic term is a real part 
of a holomorphic function and thus harmonic.  
Since $f$ is of exponential type, there exist
constants $A$ and $B$ such that we have for every 
$\xi,\eta\in \R^n$ and every $z\in \C$ that
$
|f(\xi+z\eta)|\leq  Be^{A|\xi+z\eta|}\leq e^{\log B+A|\xi|+A|\eta||z|}$.  
By estimating the logarithmic term by $\varepsilon|\eta||z|$ 
we see that for every $\varepsilon>0$ 
there exists a constant $c_\varepsilon$ such that 
$v(z)\leq c_\varepsilon+(A+\varepsilon)|\eta||z|$ for every
$z\in \C_+$.    The estimate 
$|\xi+x\eta|\leq \sqrt 2 ||\xi|-ix|\eta||$ implies
that $\varlimsup_{\C_+\ni z\to x}v(z)\leq 0$.
By the lemma, $v(z)\leq A|\eta|\Im z$.
The inequality  $v(i)\leq A|\eta|$ implies
$$
|f(\zeta)|\leq 2^{N} C(1+|\zeta|)^Ne^{A|\Im \zeta|}, \qquad \zeta\in \C^n,
$$
and Theorem 5.2 implies that $f$ is the Fourier-Laplace
transform of a distri\-bution $u$ with support in the ball
$\{x\in \R^n \,;\, |x|\leq A\}$.  As in (ii) it follows
that $\chsupp u\subseteq K_{E,\sigma}$.

(iv) Let $u\in L^2(\R)$ be the inverse Fourier-transform 
of the restriction of $f$ to $\R^n$.  Take $0\leq \varphi\in
C_0^\infty(\R^n)$ with support in the closed unit ball, 
$\|\varphi\|_{L^1(\R^n)}=1$, and 
for every $\delta>0$ define
$\varphi_\delta$ by $\varphi_\delta(x)=\varphi(x/\delta)/\delta^n$.
Then $\|u*\varphi_\delta\|_{L^1(\R)}\leq
\delta^{-\frac n2}\|u\|_{L^2(\R^n)}\|\varphi\|_{L^2(\R^n)}$, 
so $u*\varphi_\delta$ is in 
$L^1(\R^n)$   and con\-sequently its Fourier transform
$f_\delta=f\widehat \varphi_\delta$ is bounded on the real axis.
We have $|\widehat \varphi_\delta(\zeta)|\leq e^{\delta|\Im\zeta|}$
for every $\zeta\in \C^n$, so $f_\delta$ is of exponential type
$\leq A+\delta$, if $f$ is of type $\leq A$.  From (iii) we conclude that 
$u*\varphi_\delta$ has support in the ball $\{x\in \R^n \,;\, 
|x|\leq A+\delta\}$ and since $u*\varphi_\delta\to u$ as $\delta\to 0$
in the sense of distributions we conclude that
$u$ has support in $\{x\in \R^n \,;\, 
|x|\leq A\}$.  Again, with the same argument as  in (ii) it follows that 
$\chsupp u\subseteq K_{E,\sigma}$.
\end{proof}

For a detailed study of growth properties 
of plurisubharmonic functions related to Fourier-Laplace  transforms
see H\"ormander and Sigurdsson  \cite{HorSig:1998}.

\section{Applications to Radon transforms}
\label{sec:6}

\noindent
Recall that the Radon transform  ${\mathcal R}u$ of 
rapidly decreasing continuous function  $u$ on  $\R^n$  
is defined by
\begin{equation*}
{\mathcal R}u(\omega,p) 
=\int_{\scalar x\omega=p} u(x)\, dm(x), \qquad \omega\in
{\mathbb S}^{n-1}, \ p\in \R, 
\end{equation*}
where  $dm$ is the Lebesgue measure in the hyperplane 
given by $\scalar x\omega=p$.   
If we fix $\omega\in {\mathbb S}^{n-1}$, take the Fourier
transform of $p\mapsto {\mathcal R}u(\omega,p)$ and apply
Fubini's theorem, then we get for $s\in \R$ 
\begin{multline}
\label{eq:6.1}
  \F_1\big({\mathcal R}u(\omega,\cdot)\big)(s)
=\int_{-\infty}^{+\infty}e^{-isp} \int_{\scalar x\omega=p} 
u(x)\, dm(x) \, dp 
\\
=\int_{-\infty}^{+\infty} \int_{\scalar x\omega=p} 
e^{-is\scalar x\omega}u(x)\, dm(x) \, dp 
=\F_nu(s\omega)=\widehat u(s\omega).
\end{multline}
Since $u$ is rapidly decreasing, we have 
$\widehat u\in C^{\infty}(\R^n)$. For $k=0,1,2,\dots$ we define
the  $k$-homogeneous complex polynomial 
$P_k$ by 
\begin{equation*}
P_k(\zeta) 
=\sum_{|\alpha|=k}\dfrac{\partial^\alpha \widehat u(0)}{\alpha!} \zeta^\alpha,
\qquad  \zeta\in \C^n.
\end{equation*}
Observe that a priori  nothing is known about the convergence of
$\sum_{k=0}^\infty P_k$.

If $E\subseteq {\mathbb S}^{n-1}$ and for
every $\omega\in E$ the function
$\R\ni p\mapsto {\mathcal R}u(\omega,p)$ 
has compact support, contained in the interval 
$[a_\omega,b_\omega]$,  then
$\F_1\big({\mathcal R}u(\omega,\cdot)\big)$
extends to an entire function on $\C$ satisfying
the estimate
\begin{equation}
\label{eq:6.2}
 \big|\F_1\big({\mathcal R}u(\omega,\cdot)\big)(z)\big|
\leq \|{\mathcal R}u(\omega,\cdot)\|_{L^1(\R)}
e^{H_\omega(\Im z)}, \qquad z\in \C, 
\end{equation}
where $H_\omega$ is the supporting function 
of the interval $[a_\omega,b_\omega]$, i.e.,
$H_\omega(t)=a_\omega t$ for $t\leq 0$ and
$H_\omega(t)=b_\omega t$ for $t\geq 0$.

If we assume that $E$ has positive homogeneous capacity
in the sense of Siciak, and the function $\sigma$
defined by $\sigma(\omega)=\max\{-a_\omega,b_\omega\}$ is
bounded above on $E$, then the assumptions of Theorem 
\ref{th:4.1} are satisfied for the function $f\colon \C E\to \C$ defined
by $f(z\omega)=\F_1\big({\mathcal R}u(\omega,\cdot)\big)(z)$.
Hence   $f$ extends to an entire function on $\C^n$ and 
it is given by the Taylor series of $\widehat u$ at the origin.
Furthermore, Theorems \ref{th:5.1} and  \ref{th:5.5} 
imply  that $u$ has compact support  contained in 
$K_{E,\sigma}=\{x\in \R^n \,;\, \scalar x\xi\leq  \Psi_{E,\sigma}^*(\xi), 
\forall \xi\in \R^n\}$.

By (\ref{eq:6.1}) and (\ref{eq:6.2}) we have
$$
i_{\widehat u}(-i\omega) =\varlimsup_{t\to +\infty}\dfrac 1t
\log|\F_1{\mathcal R}u(\omega,\cdot)(-it)|\leq -a_\omega,
\qquad \omega\in E,
$$
and
$$
i_{\widehat u}(i\omega) =\varlimsup_{t\to +\infty}\dfrac 1t
\log|\F_1{\mathcal R}u(\omega,\cdot)(it)|\leq b_\omega,
\qquad \omega\in E.
$$ 
By Sigurdsson \cite{Sig:1986}, Th.~2.1.1, 
and by Wiegerinck  \cite{Wie:1984}, Th.~2,
the supporting function  $H$ of $\chsupp u$ is given by 
$H(\omega)=i_{\widehat u}^*(i\omega)$ for every $\omega\in {\mathbb
  S}^{n-1}$ and  we have
$$
H(\omega)
=\varlimsup_{{\mathbb S}^{n-1}\ni \tilde \omega\to \omega}i_{\widehat
  u}(i\tilde \omega), \qquad \omega\in {\mathbb S}^{n-1}.
$$
If we assume that $\omega\mapsto a_\omega$ and $\omega\mapsto
b_\omega$ are lower and upper semi-continuous, respectively,
then this  implies that for every interior point $\omega$ of $E$
$$
H(-\omega)\leq -a_\omega \quad \text{ and } \quad 
H(\omega)\leq b_\omega.
$$
If we assume that $E$ is the closure of its interior in ${\mathbb
  S}^{n-1}$, then these in\-equal\-ities hold at every point in $E$,
and we conclude that for every $x\in \chsupp u$ and every $\omega\in
E$ we have 
$$
a_\omega\leq -H(-\omega)\leq \scalar x\omega \leq H(\omega)\leq b_\omega.
$$
We summarize our argument in

\begin{theorem}
  \label{th:6.1} 
Let $u\in C(\R^n)$ be rapidly decreasing
and $E$ be a non-empty compact subset of ${\mathbb S}^{n-1}$,
which has positive homogeneous capacity.
Assume that the function  
$\R\ni p\mapsto {\mathcal R}u(\omega,p)$
has support contained in the closed bounded
interval $[a_\omega,b_\omega]$ for every
$\omega$ in $E$, where 
the  functions $\omega\mapsto a_\omega$ 
and $\omega\mapsto b_\omega$ are bounded from below  
and above, respectively, and set
$\sigma(\omega)=\max\{-a_\omega,b_\omega\}$ for $\omega\in E$.
Then the support of $u$ is contained in the compact
convex set 
\begin{equation*}
\{x\in \R^n \,;\, \scalar x\omega\leq  \Psi_{E,\sigma}^*(\omega), 
\forall \omega\in {\mathbb S}^{n-1}\}.
\end{equation*}
If $E$  is the closure of its relative interior in ${\mathbb
  S}^{n-1}$ and the functions $\omega\mapsto a_\omega$ 
and $\omega\mapsto b_\omega$ are lower and upper
semi-continuous, respectively,
then the support of $u$ is contained in the compact
convex set 
\begin{equation*}
\{x\in \R^n \,;\, \scalar x\omega\leq  \Psi_{E,\sigma}^*(\omega), 
\forall \omega\in {\mathbb S}^{n-1}, 
 a_\omega\leq \scalar x\omega \leq b_\omega,
\forall \omega\in E\}.
\end{equation*}
\end{theorem}

If $K$ satisfies  the assumptions in Helgason's theorem, i.e.,
$K$ is compact, convex, and  ${\mathcal R}(\omega,p)=0$ for every $(\omega,p)$ 
such that the hyperplane defined by the equation $\scalar x\omega=p$ does
not intersect $K$, 
and we define for $\omega\in {\mathbb S}^{n-1}$
$$
a_\omega=\inf_{x\in K}\scalar x\omega=-H_K(-\omega)
\quad \text{ and } \quad 
b_\omega=\sup_{x\in K}\scalar x\omega=H_K(\omega),
$$
then $\omega\mapsto a_\omega$ and $\omega\mapsto b_\omega$ are
continuous on ${\mathbb S}^{n-1}$,
$a_{-\omega}=-b_\omega$,
$b_{-\omega}=-a_\omega$,  the function
$p\mapsto {\mathcal R}u(\omega,p)$ has support in 
$[a_\omega,b_\omega]$, and 
$$K=\{x\in \R^n\,;\, a_\omega \leq  \scalar x\omega \leq b_\omega, \forall
\omega\in{\mathbb S}^{n-1}\}.
$$ 
Hence Theorem \ref{th:6.1} is a generalization of Helgason's support
theorem. 
It is also a generalization of a theorem of 
Wiegerinck \cite{Wie:1985} which states that 
a rapidly decreasing function $u$
has compact support under the 
assumption that $p\mapsto {\mathcal R}u(\omega,p)$ 
decreases exponentially as $|p|\to+\infty$ for every $\omega\in
{\mathbb S}^{n-1}$ and  has compact support for every
$\omega$ in an open subset $E$ of ${\mathbb S}^{n-1}$.  
His proof is based on a
very interesting result of   
Korevaar and Wiegerinck
\cite{KorWie:1985,WieKor:1985} 
on representation of mixed derivatives of functions in terms of 
higher order directional derivatives.

{\small
\bibliographystyle{plain}
\bibliography{bibref}

\noindent
Department of Mathematical Sciences, \\
Chalmers University of Technology and University of Gothenburg\\
SE-412 96 G\"oteborg, SWEDEN. \\
bergh@chalmers.se\\
Department of Mathematics, \\
School of Engineering and Natural Sciences,\\
University of Iceland,\\
IS-107 Reykjav\'ik, ICELAND. \\
ragnar@hi.is
}%\small

\end{document}